\documentclass[11pt]{amsart}
\usepackage{txfonts,color}
\usepackage{amssymb}
\theoremstyle{plain}

\newtheorem{theorem}{Theorem}[section]
\newtheorem{lemma}[theorem]{Lemma}
\newtheorem{proposition}[theorem]{Proposition}
\newtheorem{corollary}[theorem]{Corollary}
\theoremstyle{remark}

\newcommand{\R}{\mathbb R}
\newcommand{\eps}{\varepsilon}
\newcommand{\dd}{\,ds}

\renewcommand{\epsilon}{\varepsilon}

\newcommand{\bef}{\begin{flushright}}
\newcommand{\eef}{\end{flushright}}

\newcommand{\eval}[2][\right]{\relax
\ifx#1\right\relax \left.\fi#2#1\rvert}

\def\beq{\begin{eqnarray*}}
\def\eeq{\end{eqnarray*}}

\begin{document}

\title[The eighth-order evolution model]{The eighth-order evolution model for closed planar curves}

\author{Yanlong Zhang,  Li Ma }

\address{Yanlong Zhang, Institute of Mathematics, Henan Academy of sciences, Chongshili 228, Zhengzhou 450046, China}
\email{ylzhang@gmail.com}

\address{Li Ma,
School of Mathematics and Physics, University of Science and Technology Beijing, Xueyuan Road 30, Haidian, Beijing 100083, China}
\email{nuslma@126.com, lma17@ustb.edu.cn}


\date{Oct. 2025; revised May 20th, 2026}

\begin{abstract}
We consider the $L^2(ds)$ gradient flow of $\frac12\int k_{ss}^2\,ds$ for smooth closed immersed planar curves. The stationary curves are precisely the multiply covered circles. We give an explicit sufficient condition under which the elastic energy is non-increasing and for each nonzero turning number, explicit small initial scale-invariant energy yields a global smooth solution with bounded positive length and exponential convergence to a fixed multiply covered circle, after reparametrisation. We note that that the qualitative rigidity and convergence results are already contained in the generalised ideal flow theory of McCoy and Wheeler. The present paper gives quantitative measure of small initial data, corrected evolution identities, and a complete argument in the case of the curvature energy $\frac12\int k_{ss}^2\,ds$.
\end{abstract}

\maketitle

\section{Introduction}

In recent years, higher order geometric flows have attracted attentions in the field of Differential geometry. In some dimension, the higher flows have a global flow with arbitrary initial data, see \cite{ML}.
Curve flows (i.e., curvature flows) are important research subjects in geometric analysis and have applications in computer graphing. Their are also models from applied science as moving curved fronts in chemistry reaction-diffusion. Higher order curvature flows are relatively new objects, which aim to minimize the curvature energy. The famous curve flow is the $L^2$-gradient flow of the length functional. If we let $k=k_\gamma$ be the curvature of planar curve $\gamma$, then we have the bending action
  $$
  B(\gamma)=\frac12\int_\gamma k_\gamma^2 ds
  $$
  where $s$ is the arc-length parameter of the curve $\gamma$. If we do variation while keeping the length fixed, we then have Euler-Lagrange equation
  $$
  2k_{ss}+k^3-\lambda k=0
  $$
  and the $L^2$ gradient flow is
  $$
  \gamma_t=-(2k_{ss}+k^3-\lambda k)\nu
  $$
  where $\nu$ is unit normal to the curve $\gamma$. We may associate other interesting functionals such that
    $$
    E_*[\gamma]=\frac12\int_\gamma k_s^2 ds
    $$
    and
    $$
    F_*[\gamma] =\frac12 \int_\gamma k_{ss}^2ds.
    $$
    In the recent interesting work of B.Andrews, etc, the authors have studied the gradient flow of the energy $E_*[\cdot]$ and have obtained a global solution and the limit provided the initial energy is small in some sense. In present work, we study the gradient flow of the energy $F_*[\cdot]$ and have obtained a global solution and the limit provided the initial energy is small in some sense.

Let $\gamma:S^1\to\R^2$ be a smooth closed immersion. We write $s$ for arc-length, $\tau=\gamma_s$ for the unit tangent, and $\nu=(-\tau_2,\tau_1)$ for the unit normal. The signed curvature $k$ is defined by $\tau_s=k\nu$. We study the energy
\begin{equation}\label{eq:energy}
 E[\gamma]=\frac12\int_\gamma k_{ss}^2\dd
\end{equation}
and its descending $L^2(ds)$ gradient flow
\begin{equation}\label{eq:flow}
 \partial_t\gamma=V\nu,\qquad
 V=-k_{s^6}-\frac12k k_{ss}^2+k k_s k_{s^3}-k^2k_{s^4}.
\end{equation}
Here $k_{s^j}=\partial_s^jk$. Because curvature contains two derivatives of the immersion, \eqref{eq:flow} is an eighth-order equation for $\gamma$. Its curvature equation has principal term $-k_{s^8}$.

The length and turning number are denoted by $L=L[\gamma]$ and $\omega=(2\pi)^{-1}\int_\gamma k\dd$. The average curvature is $\bar k=2\pi\omega/L$. Along a smooth solution, the turning number is constant. We retain the scale-invariant quantity
\begin{equation}\label{eq:epsilon}
 \eps(t)=L(t)^5\int_{\gamma(t)}k_{ss}^2\dd.
\end{equation}
The elastic energy is $E_0[\gamma]=\frac12\int_\gamma k^2\dd$.

Andrews, McCoy, Wheeler and Wheeler \cite{AMWW} considered the energy $\frac12\int k_s^2\dd$ and its sixth-order flow. McCoy and Wheeler \cite{MW} treat \eqref{eq:flow} within their general higher-order theory. Thus the stationary classification and small-data convergence below are not claimed as new qualitative results except quantitative aspects. We use their gradient inequality and smooth local theory, with the precise interfaces stated in Section~4. The curvature interpolation estimates are developed explicitly; the corresponding approach to continuation for elastic flows appears in \cite{DKS}.

For the stationary case, we know the below.

\begin{theorem}\label{thm:stationary}
A smooth closed immersed planar curve satisfies $V=0$, with $V$ as in \eqref{eq:flow}, if and only if it is a round circle covered a nonzero integer number of times. In particular, there is no stationary closed immersion with turning number zero.
\end{theorem}

For every initial smooth closed immersion, parametrised with constant speed, there is a unique local solution to the flow \eqref{eq:flow}. See Theorem \ref{thm:local} and \cite{MW}.

To obtain the global existence of the flow, we need some uniform small constants.

For $\omega\in\mathbb Z$, We define $c_\omega$ in \eqref{eq:gradient} and
we choose the constant $\eps_1$ in \eqref{eq:epsilon1}, which will be needed for the smallness argument later.
Let $\eps_*(\omega)$ be the unique positive solution of
\begin{equation}\label{eq:threshold}
 \frac{\eps_*}{(2\pi)^5}
 \left[\frac{113}{2}+12\left(|\omega|+
                \frac{\sqrt{\eps_*}}{(2\pi)^{5/2}}\right)^2\right]=\frac4{13}.
\end{equation}
Define
\begin{equation}\label{eq:epsilon2}
 \eps_2=\frac{\eps_1}{2}
 \left(\frac{(2\pi\omega)^2}
 {(2\pi\omega)^2+\eps_1/(2\pi)^4}\right)^5.
\end{equation}
Both constants are positive and depend only on $\omega$ and the fixed choice of $c_\omega$.

Then for the existence of the global flow and its asymptotic, we have the following result.
\begin{theorem}\label{thm:global}
Let $\gamma_0:S^1\to\R^2$ be a smooth closed immersion with turning number $\omega\ne0$ and $\eps(0)\le\eps_2$, where \eqref{eq:threshold}, \eqref{eq:gradient}, \eqref{eq:epsilon1} and \eqref{eq:epsilon2} determine the constants. Then its geometric flow \eqref{eq:flow} exists smoothly for all $t\ge0$ and is unique modulo reparametrisation. The length satisfies \eqref{eq:two-sided-length}, $E_0$ is non-increasing, and there are $C,c>0$, depending only on $\omega$ and the initial curve, such that
\begin{equation}\label{eq:exponential-energy}
 \eps(t)+\int k_{ss}^2\dd\le Ce^{-ct}.
\end{equation}
For every integer $j\ge0$, there is $C_j$, depending additionally on $j$, such that
\begin{equation}\label{eq:exponential-curvature}
 \|\partial_s^j(k-\bar k)\|_\infty\le C_j e^{-ct}.
\end{equation}
After reducing $c$ once, the same positive $c$ can be used for all $j$.
\end{theorem}

The explicit number $\eps_2$ makes our result different from the main results in \cite{MW}. Our result is a little more quantitative theorem.
We shall present the explicit monotonicity threshold in Lemma~\ref{lem:monotonicity} which is distinct from the global initial-data threshold in Theorem~\ref{thm:global}. The latter also depends on the constant in the cited gradient inequality. In particular, numerical values of the former are not numerical values of the latter.

All norms and integrals below use the arc-length measure of the current curve, unless another measure is specified. Constants denoted by $C$ may change from line to line; their dependence is stated where they are used.

 We organize the paper as follows. In section \ref{sect2}, we recall some useful formulae. Theorem \ref{thm:stationary} is proved in section \ref{sect3} and Theorem \ref{thm:global} in section \ref{sect4}. We make some discussion about the limitation of our results in section \ref{sect5}.

\section{First variation and evolution identities}\label{sect2}

\subsection{The gradient flow}

For a smooth normal variation with velocity $\varphi\nu$, differentiation of $ds=|\gamma_x|\,dx$ and of the Frenet equations gives
\begin{equation}\label{eq:variation}
 \partial_t ds=-k\varphi\,ds,\qquad
 [\partial_t,\partial_s]=k\varphi\partial_s,\qquad
 \partial_t k=\varphi_{ss}+k^2\varphi.
\end{equation}
To verify the curvature identity, the commutator first gives $\partial_t\tau=\varphi_s\nu$ and $\partial_t\nu=-\varphi_s\tau$. Differentiating $\tau_s=k\nu$ then gives its normal component as $\partial_tk=\varphi_{ss}+k^2\varphi$.

Applying the commutator once and twice to $k$ yields
\begin{align*}
 \partial_tk_s&=\varphi_{s^3}+3kk_s\varphi+k^2\varphi_s,\\
 \partial_tk_{ss}&=\varphi_{s^4}+3k_s^2\varphi+4kk_{ss}\varphi
                 +5kk_s\varphi_s+k^2\varphi_{ss}.
\end{align*}
All boundary terms vanish because the curve and its arc-length derivatives are periodic. Consequently,
\begin{align*}
 \frac{d}{dt}E[\gamma]
 &=\int k_{ss}\bigl(\varphi_{s^4}+3k_s^2\varphi+4kk_{ss}\varphi
                  +5kk_s\varphi_s+k^2\varphi_{ss}\bigr)\dd
   -\frac12\int kk_{ss}^2\varphi\dd\\
 &=\int\varphi\left(k_{s^6}+3k_s^2k_{ss}+\frac72kk_{ss}^2
                   -(5kk_sk_{ss})_s+(k^2k_{ss})_{ss}\right)\dd\\
 &=\int\varphi\left(k_{s^6}+\frac12kk_{ss}^2-kk_sk_{s^3}
                   +k^2k_{s^4}\right)\dd.
\end{align*}
Thus the velocity in \eqref{eq:flow} is the negative gradient of $E$, and
\begin{equation}\label{eq:dissipation}
 \frac{d}{dt}E[\gamma(t)]=-\int V^2\dd.
\end{equation}
In particular, the derivative of $\int k_{ss}^2\dd$ is $-2\int V^2\dd$.

\subsection{Length, elastic energy and turning number}

\begin{lemma}\label{lem:length}
For every smooth solution of \eqref{eq:flow},
\begin{equation}\label{eq:length}
 L'=-\int k_{s^3}^2\dd+\frac92\int k^2k_{ss}^2\dd
                  -\frac83\int k_s^4\dd.
\end{equation}
\end{lemma}
\begin{proof}
By \eqref{eq:variation}, $L'=-\int kV\dd$. The required integrations by parts are
\begin{align*}
 \int kk_{s^6}\dd&=-\int k_{s^3}^2\dd,\\
 \int k^2k_sk_{s^3}\dd
 &=-\int k^2k_{ss}^2\dd-2\int kk_s^2k_{ss}\dd
  =-\int k^2k_{ss}^2\dd+\frac23\int k_s^4\dd,\\
 \int k^3k_{s^4}\dd
 &=3\int k^2k_{ss}^2\dd+6\int kk_s^2k_{ss}\dd
  =3\int k^2k_{ss}^2\dd-2\int k_s^4\dd.
\end{align*}
Substitution into $-\int kV\dd$ gives \eqref{eq:length}.
\end{proof}

\begin{lemma}\label{lem:elastic}
For every smooth solution of \eqref{eq:flow},
\begin{align}\label{eq:elastic}
 \frac{dE_0}{dt}
 ={}&-\int k_{s^4}^2\dd-\frac{17}{2}\int k_s^2k_{ss}^2\dd
       +13\int kk_sk_{ss}k_{s^3}\dd\notag\\
    &+\frac52\int k^2k_{s^3}^2\dd
       +12\int k^2k_s^4\dd-\frac{13}{4}\int k^4k_{ss}^2\dd.
\end{align}
\end{lemma}
\begin{proof}
Differentiating the integral with its moving measure gives
\begin{equation}\label{eq:elastic-first}
 \frac{dE_0}{dt}=\int\left(k_{ss}+\frac12k^3\right)V\dd.
\end{equation}
The identities needed to expand this expression are
\begin{align*}
 \int kk_{ss}^3\dd
   &=-\int k_s^2k_{ss}^2\dd-2\int kk_sk_{ss}k_{s^3}\dd,\\
 \int k^2k_{ss}k_{s^4}\dd
   &=-2\int kk_sk_{ss}k_{s^3}\dd-\int k^2k_{s^3}^2\dd,\\
 \int k^3k_{s^6}\dd
   &=18\int k_s^2k_{ss}^2\dd-18\int kk_sk_{ss}k_{s^3}\dd
                                  -3\int k^2k_{s^3}^2\dd,\\
 \int k^4k_sk_{s^3}\dd
   &=4\int k^2k_s^4\dd-\int k^4k_{ss}^2\dd,\\
 \int k^5k_{s^4}\dd
   &=-20\int k^2k_s^4\dd+5\int k^4k_{ss}^2\dd.
\end{align*}
For the third identity, integrate twice to obtain
$$\int k^3k_{s^6}\dd=\int(6kk_s^2+3k^2k_{ss})k_{s^4}\dd.$$
Then use the second identity and
$$\int kk_s^2k_{s^4}\dd=3\int k_s^2k_{ss}^2\dd
                           -2\int kk_sk_{ss}k_{s^3}\dd.$$
The other displayed identities follow by one or two integrations by parts. Expanding \eqref{eq:elastic-first} with these identities gives \eqref{eq:elastic}.
\end{proof}

\begin{lemma}\label{lem:turning}
The turning number $\omega=(2\pi)^{-1}\int k\dd$ is preserved by every smooth solution of \eqref{eq:flow}.
\end{lemma}
\begin{proof}
Equations \eqref{eq:variation} give
$$2\pi\omega'=\int(k_t-k^2V)\dd=\int V_{ss}\dd=0.$$
\end{proof}

\section{Stationary curves}\label{sect3}

Here is the proof of Theorem \ref{thm:stationary}.
\begin{proof}
Suppose $V=0$. Retaining the quantities used in the first-integral calculation, write
$$M=k_{s^5},\qquad N=kk_{s^4}-k_sk_{s^3}+\frac12k_{ss}^2.$$
Their derivatives satisfy
$$M_s=-kN,\qquad N_s=kM.$$
The first identity is exactly $V=0$, and the second follows by differentiating $N$ and cancelling the two terms $k_sk_{s^4}$ and the two terms $k_{ss}k_{s^3}$. No division by $k$ is involved.

On the arclength interval $[0,L]$, choose a real tangent angle $\theta$ with $\tau=(\cos\theta,\sin\theta)$ and $\theta_s=k$. Identifying $\R^2$ with $\mathbb C$, the preceding equations imply
$$\bigl((M+iN)e^{-i\theta}\bigr)_s=0.$$
Closedness gives $\int_0^L e^{i\theta}\dd=\int_0^L\gamma_s\dd=0$, so $\int_0^L(M+iN)\dd=0$. On the other hand,
$$\int_0^L N\dd
 =\int_0^L\left(kk_{s^4}-k_sk_{s^3}+\frac12k_{ss}^2\right)\dd
 =\frac52\int_0^L k_{ss}^2\dd.$$
It follows that $k_{ss}=0$. Periodicity then gives $k_s=0$, so $k$ is constant. A constant zero curvature would make $\gamma$ a straight line and contradict closedness of an immersion. Integrating the Frenet equations with constant nonzero curvature yields a circle, and closure requires $kL=2\pi\omega$ with $\omega\ne0$. Conversely, all curvature derivatives vanish on a multiply covered circle, so $V=0$.
\end{proof}

The same classification has a short variational verification: $E[\lambda\gamma]=\lambda^{-5}E[\gamma]$. At a stationary curve the first variation under dilation vanishes, and therefore $E[\gamma]=0$. The argument above also repairs the first-integral proof without an analyticity assumption or a rescaling of its constant.

\section{Existence and convergence}

\subsection{Local existence and the parametrisation}

For a constant-speed parametrisation $x\in\R/\mathbb Z$, we have $ds=L\,dx$. To fix parametrisation, let $\alpha$ be the unique periodic function satisfying
\begin{equation}\label{eq:alpha}
 \alpha_s=kV-\frac1L\int kV\dd,\qquad \int\alpha\dd=0,
\end{equation}
and consider
\begin{equation}\label{eq:gauge}
 \gamma_t=V\nu+\alpha\tau.
\end{equation}
The right-hand side of \eqref{eq:alpha} has zero mean, which gives a periodic primitive; the mean condition fixes its additive constant. Differentiation gives
$$\partial_t|\gamma_x|=(\alpha_s-kV)|\gamma_x|
                =-\frac1L\left(\int kV\dd\right)|\gamma_x|.$$
Thus a constant-speed initial parametrisation remains constant-speed. The tangential term changes only parametrisation, and all the preceding integral identities remain valid. For example, its contribution to the derivative of $\int f\dd$ is $\int(\alpha f_s+\alpha_s f)\dd=0$.

We have the standard local existence of the flow, see the $m=2$ case of \cite[Theorems 5.2 and 5.3]{MW}.
\begin{theorem}\label{thm:local}
Every smooth closed immersion, parametrised with constant speed, has a unique maximal smooth solution of \eqref{eq:gauge} on $[0,T_{\max})$. If $T_{\max}<\infty$, two-sided positive length bounds together with bounds for every curvature Sobolev norm imply smooth continuation past $T_{\max}$.
\end{theorem}

The hypotheses of the cited results match our smooth periodic immersion and constant-speed parametrisation; their normal velocity is exactly $V$. A solution of the original equation \eqref{eq:flow} is recovered by solving the parameter transport equation with velocity $-\alpha/L$ on $\R/\mathbb Z$. Conversely, a smooth solution of \eqref{eq:flow} can be put in the gauge \eqref{eq:gauge}. We henceforth work with this representative and state geometric uniqueness modulo reparametrisation.

\subsection{Small energy, length bounds and continuation}\label{sect4}

\begin{lemma}\label{lem:poincare}
If $f$ is a smooth periodic function on an interval of length $L$ and $\int f\dd=0$, then
\begin{equation}\label{eq:poincare}
 \|f\|_2^2\le\frac{L^2}{(2\pi)^2}\|f_s\|_2^2,\qquad
 \|f\|_\infty^2\le\frac{L}{2\pi}\|f_s\|_2^2.
\end{equation}
For a smooth closed immersion, these inequalities imply
\begin{align}\label{eq:curvature-bounds}
 L\int(k-\bar k)^2\dd&\le\frac{\eps}{(2\pi)^4},\notag\\
 L\|k-\bar k\|_\infty&\le\frac{\sqrt\eps}{(2\pi)^{3/2}},\notag\\
 (2\pi\omega)^2\le L\int k^2\dd
       &\le(2\pi\omega)^2+\frac{\eps}{(2\pi)^4}.
\end{align}
\end{lemma}
\begin{proof}
Expand $f(s)=\sum_{n\ne0}\widehat f_n e^{2\pi ins/L}$. Parseval gives the first inequality. Cauchy--Schwarz gives
$$|f(s)|^2\le\left(\sum_{n\ne0}\frac{L^2}{4\pi^2n^2}\right)
                \sum_{n\ne0}\frac{4\pi^2n^2}{L^2}|\widehat f_n|^2
          =\frac{L}{12}\|f_s\|_2^2
          \le\frac{L}{2\pi}\|f_s\|_2^2.$$
Apply these inequalities successively to $k-\bar k$ and $k_s$, using $\int k_s\dd=0$. Finally, $\int k^2\dd=L\bar k^2+\int(k-\bar k)^2\dd$, which proves the last line.
\end{proof}

\begin{lemma}\label{lem:monotonicity}
For $\omega\in\mathbb Z$, let $\eps_*(\omega)$ be the unique positive solution of \eqref{eq:threshold}.
Every smooth solution of \eqref{eq:flow} satisfies
\begin{align}\label{eq:monotonicity}
 \frac{dE_0}{dt}\le
 -\left[\frac4{13}-\frac{113\eps}{2(2\pi)^5}
 -\frac{12\eps}{(2\pi)^5}\left(|\omega|+
                       \frac{\sqrt\eps}{(2\pi)^{5/2}}\right)^2\right]
 \int k_{s^4}^2\dd.
\end{align}
Consequently $E_0$ is non-increasing on any interval where $\eps\le\eps_*(\omega)$. On an interval where $\eps\le\eps_*(\omega)/2$, it satisfies $E_0'\le-(2/13)\int k_{s^4}^2\dd$.
\end{lemma}
\begin{proof}
Integration by parts gives
$$\int k^2k_{s^3}^2\dd=-\int k^2k_{ss}k_{s^4}\dd
                          -2\int kk_sk_{ss}k_{s^3}\dd.$$
By Cauchy--Schwarz and $2ab\le\frac12a^2+2b^2$,
\begin{align*}
 \int k^2k_{s^3}^2\dd
 &\le\left(\int k_{s^4}^2\dd\int k^4k_{ss}^2\dd\right)^{1/2}
       +2\left(\int k_s^2k_{ss}^2\dd\int k^2k_{s^3}^2\dd\right)^{1/2}\\
 &\le\left(\int k_{s^4}^2\dd\int k^4k_{ss}^2\dd\right)^{1/2}
       +\frac12\int k^2k_{s^3}^2\dd+2\int k_s^2k_{ss}^2\dd.
\end{align*}
Thus
$$\int k^2k_{s^3}^2\dd\le
 2\left(\int k_{s^4}^2\dd\int k^4k_{ss}^2\dd\right)^{1/2}
 +4\int k_s^2k_{ss}^2\dd.$$
The same integral identity rewrites the two mixed terms in \eqref{eq:elastic} as
$$13\int kk_sk_{ss}k_{s^3}\dd+\frac52\int k^2k_{s^3}^2\dd
 =8\int kk_sk_{ss}k_{s^3}\dd-\frac52\int k^2k_{ss}k_{s^4}\dd.$$
Using $8ab\le\frac14a^2+64b^2$, we have
$$8\left|\int kk_sk_{ss}k_{s^3}\dd\right|
 \le\frac14\int k^2k_{s^3}^2\dd+64\int k_s^2k_{ss}^2\dd.$$
Substitution into \eqref{eq:elastic} now gives
\begin{align*}
 E_0'\le{}&-\int k_{s^4}^2\dd+\frac{113}{2}\int k_s^2k_{ss}^2\dd
            +12\int k^2k_s^4\dd\\
          &+3\left(\int k_{s^4}^2\dd\int k^4k_{ss}^2\dd\right)^{1/2}
            -\frac{13}{4}\int k^4k_{ss}^2\dd\\
 \le{}&-\frac4{13}\int k_{s^4}^2\dd
       +\frac{113}{2}\int k_s^2k_{ss}^2\dd+12\int k^2k_s^4\dd.
\end{align*}
The last inequality is the scalar inequality $3\sqrt{ab}-(13/4)b\le(9/13)a$ for $a,b\ge0$.

From \eqref{eq:poincare},
\begin{align*}
 \int k_s^2k_{ss}^2\dd
 &\le\frac{L}{2\pi}\left(\int k_{ss}^2\dd\right)^2
 \le\frac{\eps}{(2\pi)^5}\int k_{s^4}^2\dd,\\
 \int k^2k_s^4\dd
 &\le\|k\|_\infty^2\|k_s\|_\infty^2\int k_s^2\dd\\
 &\le\frac{\eps}{(2\pi)^5}
       \left(|\omega|+\frac{\sqrt\eps}{(2\pi)^{5/2}}\right)^2
       \int k_{s^4}^2\dd.
\end{align*}
These prove \eqref{eq:monotonicity}. The left-hand side of \eqref{eq:threshold}, as a function of $\eps_*$, is continuous, strictly increasing from zero to infinity. The positive root therefore exists and is unique. At $\eps\le\eps_*/2$, this function is at most half its value at $\eps_*$, giving the final assertion.
\end{proof}

\begin{corollary}\label{cor:lower-length}
If $\omega\ne0$ and $\eps(t)\le\eps_*(\omega)$ throughout a smooth existence interval, then
\begin{equation}\label{eq:lower-length}
 L(t)\ge\frac{(2\pi\omega)^2}{2E_0[\gamma_0]}>0.
\end{equation}
\end{corollary}
\begin{proof}
Use $(2\pi\omega)^2\le L\int k^2\dd=2LE_0(t)$ and Lemma~\ref{lem:monotonicity}.
\end{proof}

\begin{lemma}\label{lem:upper-length}
For $\sigma>0$, we define
\begin{equation}\label{eq:C-length}
 C(\sigma,\omega)=\frac92\left(\frac{\sigma}{(2\pi)^{3/2}}
                              +2\pi|\omega|\sqrt\sigma\right)^2.
\end{equation}
On any smooth existence interval on which $\eps(t)\le\sigma$,
\begin{equation}\label{eq:upper-length}
 L'\le\frac{C(\sigma,\omega)}{\sigma}\eps(t)L^{-7},\qquad
 L(t)^8\le L(0)^8+8C(\sigma,\omega)t.
\end{equation}
\end{lemma}
\begin{proof}
Drop the two nonpositive terms in \eqref{eq:length} and apply \eqref{eq:curvature-bounds}:
$$L'\le\frac92\|k\|_\infty^2\int k_{ss}^2\dd
 \le\frac92\left(2\pi|\omega|+\frac{\sqrt\sigma}{(2\pi)^{3/2}}\right)^2
              \eps(t)L^{-7}.$$
This is the first inequality, and $\eps\le\sigma$ gives the second after differentiating $L^8$.
\end{proof}

Fix $\omega\ne0$. There exists a uniform $c_\omega>0$ such that every smooth closed immersion of turning number $\omega$ satisfies
\begin{equation}\label{eq:gradient}
 \int V^2\dd\ge c_\omega L^{-8}\int k_{ss}^2\dd.
\end{equation}
For the proof, we refer to the $m=2$ case of \cite[Theorem 1.4]{MW}.
These are exactly its required hypotheses; its velocity is $V$, and its energy is \eqref{eq:energy}. We absorb the factor $1/2$ into $c_\omega$, which is not numerically specified but can be estimated in principle.

\begin{proposition}\label{prop:ratio}
Define
\begin{align}\label{eq:epsilon1}
 \eps_1=\min\left\{\frac{\eps_*(\omega)}2,\ 1,\
 \frac{c_\omega(2\pi\omega)^4}
 {50\left[1+\frac14\left(2\pi|\omega|+(2\pi)^{-3/2}\right)^4
               \left((2\pi\omega)^2+(2\pi)^{-4}\right)\right]}\right\}.
\end{align}
Let $\gamma$ be a smooth solution of \eqref{eq:flow} with turning number $\omega\ne0$. On any interval where $\eps\le\eps_1$, the ratio
\begin{equation}\label{eq:ratio}
 Q(t)=\frac{\int k_{ss}^2\dd}{\left(\int k^2\dd\right)^5}
\end{equation}
satisfies
\begin{equation}\label{eq:ratio-decay}
 Q'\le-\frac{c_\omega}{2}L^{-8}Q,\qquad
 Q(t)\le Q(0)\exp\left(-\frac{c_\omega}{2}\int_0^t L(r)^{-8}\,dr\right)
\end{equation}
when the interval starts at $t=0$.
\end{proposition}
\begin{proof}
Differentiate \eqref{eq:ratio} using \eqref{eq:dissipation} and \eqref{eq:elastic-first}:
\begin{align*}
 Q'={}&-\frac{2\int V^2\dd}{(\int k^2\dd)^5}
 -\frac{10\int k_{ss}^2\dd}{(\int k^2\dd)^6}
            \int\left(k_{ss}+\frac12k^3\right)V\dd\\
 \le{}&-\frac{\int V^2\dd}{(\int k^2\dd)^5}
 +\frac{25(\int k_{ss}^2\dd)^2}{(\int k^2\dd)^7}
            \int\left(k_{ss}+\frac12k^3\right)^2\dd.
\end{align*}
Here Cauchy--Schwarz and $10ab\le a^2+25b^2$ were applied after factoring out $(\int k^2\dd)^{-5}$.

The mixed term is nonpositive:
$$\int k^3k_{ss}\dd=-3\int k^2k_s^2\dd.$$
Thus \eqref{eq:curvature-bounds} gives
\begin{align*}
 \int\left(k_{ss}+\frac12k^3\right)^2\dd
 &\le\int k_{ss}^2\dd+\frac14\|k\|_\infty^4\int k^2\dd\\
 &\le L^{-5}\left[\eps+\frac14
       \left(2\pi|\omega|+\frac{\sqrt\eps}{(2\pi)^{3/2}}\right)^4
       \left((2\pi\omega)^2+\frac{\eps}{(2\pi)^4}\right)\right].
\end{align*}
Using $L\int k^2\dd\ge(2\pi\omega)^2$, the positive term in the estimate for $Q'$ is at most
\begin{align*}
 \frac{25\eps}{(2\pi\omega)^4}
 \left[\eps+\frac14\left(2\pi|\omega|+
              \frac{\sqrt\eps}{(2\pi)^{3/2}}\right)^4
              \left((2\pi\omega)^2+\frac{\eps}{(2\pi)^4}\right)\right]L^{-8}Q.
\end{align*}
For $\eps\le\eps_1\le1$, the coefficient is at most $c_\omega/2$ by \eqref{eq:epsilon1}. Equation \eqref{eq:gradient} bounds the negative term by $-c_\omega L^{-8}Q$. This proves the differential inequality. Multiplication by the integrating factor gives the second assertion, including the case $Q(0)=0$.
\end{proof}

\begin{corollary}\label{cor:barrier}
Suppose a maximal smooth solution has turning number $\omega\ne0$ and $\eps(0)\le\eps_2$. Then, throughout $[0,T_{\max})$,
\begin{equation}\label{eq:barrier}
 \eps(t)\le\frac{\eps_1}{2}
 \exp\left(-\frac{c_\omega}{2}\int_0^t L(r)^{-8}\,dr\right)<\eps_1,
\end{equation}
and
\begin{equation}\label{eq:two-sided-length}
 \frac{(2\pi\omega)^2}{2E_0[\gamma_0]}
 \le L(t)\le L(0)\exp\left(\frac{C(\eps_1,\omega)}{c_\omega}\right).
\end{equation}
\end{corollary}
\begin{proof}
Because $\eps(0)\le\eps_2<\eps_1$, there is an initial interval where $\eps<\eps_1$. On every such interval,
\begin{align*}
 \eps(t)&=\left(L\int k^2\dd\right)^5Q(t)\\
 &\le\left((2\pi\omega)^2+\frac{\eps_1}{(2\pi)^4}\right)^5
       \frac{\eps(0)}{(2\pi\omega)^{10}}
       \exp\left(-\frac{c_\omega}{2}\int_0^t L(r)^{-8}\,dr\right).
\end{align*}
The initial ratio bound follows from $L(0)\int k(0)^2\dd\ge(2\pi\omega)^2$. Definition \eqref{eq:epsilon2} gives \eqref{eq:barrier}. A first time at which $\eps=\eps_1$ would contradict this estimate by continuity. Hence it holds throughout the maximal smooth existence interval.

Lemma~\ref{lem:monotonicity} and Corollary~\ref{cor:lower-length} give the lower bound. For the upper bound, Lemma~\ref{lem:upper-length} gives
\begin{align*}
 \frac{d}{dt}\log L
 &\le\frac{C(\eps_1,\omega)}{\eps_1}\eps(t)L(t)^{-8}\\
 &\le\frac{C(\eps_1,\omega)}2L(t)^{-8}
       \exp\left(-\frac{c_\omega}{2}\int_0^tL(r)^{-8}\,dr\right).
\end{align*}
The integral of the last line from $0$ to $t$ is at most $C(\eps_1,\omega)/c_\omega$, because the derivative of its exponential factor is $-(c_\omega/2)L(t)^{-8}$ times that factor. This proves \eqref{eq:two-sided-length} without assuming global existence.
\end{proof}

\begin{lemma}\label{lem:higher-norms}
Suppose a smooth solution has $0<L_-\le L(t)\le L_+<\infty$ and $\int k^2\dd\le 2E_0[\gamma_0]$ on its existence interval. For each integer $\ell\ge1$, there is a constant $C$, depending only on $\ell$, $L_-$ and $E_0[\gamma_0]$, such that
\begin{equation}\label{eq:higher-energy}
 \frac{d}{dt}\int k_{s^\ell}^2\dd+\int k_{s^{\ell+4}}^2\dd\le C.
\end{equation}
Consequently, each $\int k_{s^\ell}^2\dd$ is uniformly bounded by a constant depending additionally on $L_+$ and its initial value.
\end{lemma}
\begin{proof}
We first record the interpolation estimate with its length dependence. For a smooth periodic scalar function $f$, integers $m\ge1$ and $0\le i\le m-1$, and $2\le p\le\infty$, one has
\begin{align}\label{eq:interpolation}
 \|\partial_s^if\|_p\le{}&C
 \|\partial_s^mf\|_2^{(i+1/2-1/p)/m}
 \|f\|_2^{1-(i+1/2-1/p)/m}\notag\\
 &+C L^{-i-1/2+1/p}\|f\|_2.
\end{align}
The constants depend only on $m,i,p$. Here is a direct verification. Rescale to length one and expand in Fourier modes. Splitting the sum for $\partial^if$ at frequencies $|n|\le N$ and $|n|>N$ gives respectively the bounds $CN^{i+1/2}\|f\|_2$ and $CN^{i+1/2-m}\|\partial^mf\|_2$. Taking $N$ to be an integer comparable to $\max\{1,(\|\partial^mf\|_2/\|f\|_2)^{1/m}\}$ proves the $p=\infty$ estimate; $f=0$ is immediate. Fourier H\"older gives $\|\partial^if\|_2\le\|\partial^mf\|_2^{i/m}\|f\|_2^{1-i/m}$. Interpolating $L^2$ and $L^\infty$, and separating whether $\|\partial^mf\|_2\ge\|f\|_2$, proves the intermediate $p$ estimates. Rescaling gives \eqref{eq:interpolation}.

The derivative of the geometric integral $\int k_{s^\ell}^2\dd$ is unchanged by a tangential velocity: that velocity transports its integrand by $\alpha\partial_s$ and contributes $\alpha_s$ to the measure, giving a periodic total derivative. We may therefore compute with the purely normal representative. For the curvature equation, \eqref{eq:variation} and \eqref{eq:flow} give $k_t=-k_{s^8}$ plus sums of products of three curvature factors with total derivative order six, and of five factors with total derivative order four. Applying $[\partial_t,\partial_s]=kV\partial_s$ repeatedly shows that $\partial_tk_{s^\ell}$ has leading term $-k_{s^{\ell+8}}$; its remaining products have respectively three factors and total derivative order $\ell+6$, or five factors and total derivative order $\ell+4$.

Differentiate $\int k_{s^\ell}^2\dd$, including $\partial_tds=-kVds$. Four integrations by parts give the leading term
$$-2\int k_{s^\ell}k_{s^{\ell+8}}\dd=-2\int k_{s^{\ell+4}}^2\dd.$$
Every remaining term is an integral of four curvature factors with total derivative order $2\ell+6$, or six factors with total derivative order $2\ell+4$. Its coefficient depends only on $\ell$. This also includes the terms coming from the moving measure.

Put $m=\ell+4$. By integration by parts, the derivatives in the four-factor terms can be distributed so that each order is at most $m-1$; for the six-factor terms each order is at most $m-2$. To see termination, whenever one factor has more than half the total derivative order, integrate one of its derivatives onto the other factors. Its order decreases, while the sum of all other orders is smaller; repeating reduces the maximal order to at most half the total. No boundary term occurs.

Apply H\"older with exponent four or six and \eqref{eq:interpolation} to each factor. The powers of $\|k_{s^m}\|_2$ in the two cases are respectively
$$\frac{2\ell+7}{\ell+4}<2,\qquad \frac{2\ell+6}{\ell+4}<2.$$
For example, a four-factor product with orders $i_1+\cdots+i_4=2\ell+6$ is bounded by
\begin{align*}
 C\|k_{s^m}\|_2^{(2\ell+7)/(\ell+4)}
  \|k\|_2^{4-(2\ell+7)/(\ell+4)}
 +C L^{-(2\ell+7)}\|k\|_2^4.
\end{align*}
Mixed terms in the product of the interpolation bounds satisfy the same estimate: split according to whether $\|k_{s^m}\|_2$ is above or below $L^{-m}\|k\|_2$. The six-factor expression has powers $(2\ell+6)/(\ell+4)$ and $6-(2\ell+6)/(\ell+4)$, with remainder $CL^{-(2\ell+6)}\|k\|_2^6$.

Young's inequality bounds each product, for every $\delta>0$, by $\delta\int k_{s^m}^2\dd$ plus a constant multiple of
$$\left(\int k^2\dd\right)^{2\ell+9}
 +L^{-(2\ell+7)}\left(\int k^2\dd\right)^2
 +L^{-(2\ell+6)}\left(\int k^2\dd\right)^3.$$
The exponent $2\ell+9$ follows from dividing the lower-norm exponent by $2$ minus the higher-norm exponent in each case. There are finitely many terms for each $\ell$; choose their $\delta$ values so that their total contribution is at most $\int k_{s^m}^2\dd$. The assumed lower length bound and elastic energy bound now prove \eqref{eq:higher-energy}.

Because $k_{s^\ell}$ has zero mean for $\ell\ge1$, four applications of \eqref{eq:poincare} give
$$\int k_{s^{\ell+4}}^2\dd\ge\left(\frac{2\pi}{L_+}\right)^8
                                  \int k_{s^\ell}^2\dd.$$
Multiplying the resulting differential inequality by $\exp((2\pi/L_+)^8t)$ gives a uniform bound for $\int k_{s^\ell}^2\dd$, including its initial value.
\end{proof}

We now prove Theorem \ref{thm:global}.
\begin{proof}
Take the maximal smooth solution supplied by Theorem~\ref{thm:local}. Corollary~\ref{cor:barrier} gives two-sided length bounds throughout its maximal interval, and Lemma~\ref{lem:monotonicity} gives $\int k^2\dd\le2E_0[\gamma_0]$. Lemma~\ref{lem:higher-norms} therefore bounds every curvature Sobolev norm uniformly. If $T_{\max}$ were finite, the continuation criterion in Theorem~\ref{thm:local} would extend the solution. Thus $T_{\max}=\infty$.

With $L_+=L(0)\exp(C(\eps_1,\omega)/c_\omega)$, equation \eqref{eq:barrier} gives
$$\eps(t)\le\frac{\eps_1}{2}\exp\left(-\frac{c_\omega}{2L_+^8}t\right).$$
The lower length bound gives the same type of estimate for $\int k_{ss}^2\dd$. Alternatively, \eqref{eq:dissipation} and \eqref{eq:gradient} give directly
$$\frac{d}{dt}\int k_{ss}^2\dd\le-\frac{2c_\omega}{L_+^8}
                                          \int k_{ss}^2\dd.$$

For higher derivatives, \eqref{eq:curvature-bounds} gives exponential decay of $\|k-\bar k\|_2^2$, while Lemma~\ref{lem:higher-norms} gives a uniform bound for $\|\partial_s^m(k-\bar k)\|_2$ for every $m\ge1$. Apply \eqref{eq:interpolation} to $f=k-\bar k$, $p=\infty$ and $m=2j+2$. The exponent of the higher norm is $(j+1/2)/(2j+2)<1/2$, so the exponent of the decaying lower norm is greater than $1/2$. The remainder decays as well because $L$ is bounded below. This proves \eqref{eq:exponential-curvature} with one common smaller rate, while allowing $C_j$ to depend on $j$.
\end{proof}

\subsection{The limiting curve}

\begin{proposition}\label{prop:limit}
Under the hypotheses of Theorem~\ref{thm:global}, the constant-speed representative of the flow converges exponentially in $C^\infty(\R/\mathbb Z;\R^2)$ to a fixed multiply covered circle $\gamma_\infty$. Its length $L_\infty$ is finite and positive, its signed curvature is $2\pi\omega/L_\infty$, and its radius is $L_\infty/(2\pi|\omega|)$. The rescaled curves $\eta=\gamma/L(t)$ converge to a circle of radius $1/(2\pi|\omega|)$. In particular, with $d\hat s=ds/L$,
\begin{equation}\label{eq:normalised-curvature}
 \int\bigl(Lk-2\pi\omega\bigr)^2\,d\hat s
 \le\frac{\eps(t)}{(2\pi)^4},\qquad
 \|Lk-2\pi\omega\|_\infty\le\frac{\sqrt{\eps(t)}}{(2\pi)^{3/2}}.
\end{equation}
\end{proposition}
\begin{proof}
Every term of $V$ contains an arclength derivative of $k$. Theorem~\ref{thm:global} therefore gives exponential decay of $V$ and all its arclength derivatives. The curvature itself is uniformly bounded. Consequently $|L'|=|\int kV\dd|\le Ce^{-ct}$, so $L(t)$ converges exponentially to a finite limit. The lower bound in \eqref{eq:two-sided-length} ensures $L_\infty>0$.

Equation \eqref{eq:alpha} gives the same decay for $\alpha$ and all its arclength derivatives. For its zeroth-order bound, its zero mean supplies a point where $\alpha=0$, and integration gives $\|\alpha\|_\infty\le L\|\alpha_s\|_\infty$. Since $ds=L\,dx$ with $L$ independent of $x$, all $x$-derivatives of $V$ and $\alpha$ also decay exponentially. The Frenet equations, the length bounds and the uniform curvature derivative bounds give uniform bounds for all $x$-derivatives of $\tau$ and $\nu$. Thus \eqref{eq:gauge} implies
$$\|\gamma_t(\cdot,t)\|_{C^j(\R/\mathbb Z)}\le C_j e^{-ct}
 \quad\hbox{for every }j\ge0.$$
Integration from $t$ to infinity proves convergence of the entire trajectory in every $C^j$ norm to one fixed map $\gamma_\infty$; no subsequence, additional translation or time-dependent rotation is needed.

The equality $|\gamma_x|=L(t)$ passes to the limit and gives $|\partial_x\gamma_\infty|=L_\infty>0$. Hence the limit is an immersion. Equation \eqref{eq:exponential-curvature} gives constant signed limiting curvature $2\pi\omega/L_\infty$. Integration of the Frenet equations identifies it as the circle of radius $L_\infty/(2\pi|\omega|)$, covered $|\omega|$ times with orientation determined by the sign of $\omega$.

Division by $L(t)\to L_\infty>0$ proves convergence of $\eta$. Finally,
$$\int(Lk-2\pi\omega)^2\,d\hat s=L\int(k-\bar k)^2\dd,$$
and \eqref{eq:curvature-bounds} proves \eqref{eq:normalised-curvature}.

For the purely normal parametrisation, the transport vector field $-\alpha/L$ and all its spatial derivatives decay exponentially. Its flow of diffeomorphisms therefore converges smoothly: differentiation of its defining ODE gives uniformly bounded derivatives and integrable time derivatives at each order. Its first derivative is the exponential of a uniformly convergent time integral, and stays bounded away from zero. Its limit is a smooth diffeomorphism. Thus the normal parametrisation converges as well, with the corresponding limiting parametrisation of the same circle.
\end{proof}

\section{The thresholds and their scope}\label{sect5}

The condition in Lemma~\ref{lem:monotonicity} is a sufficient condition, not an optimal threshold. The root in \eqref{eq:threshold} can be computed without any unspecified interpolation constant. Its approximate values are
\begin{center}
\begin{tabular}{c|rrrr}
$|\omega|$ & $1$ & $2$ & $3$ & $5$\\\hline
$\eps_*(\omega)$ & $42.957256$ & $28.131792$ & $17.977278$ & $8.369341$
\end{tabular}
\end{center}
An elementary sufficient condition that avoids solving \eqref{eq:threshold} is
\begin{equation}\label{eq:algebraic-threshold}
 \eps(t)\le\frac{8(2\pi)^5}
                 {13\,[113+24(|\omega|+1)^2]}.
\end{equation}
The right-hand side is smaller than $(2\pi)^5$, so $\sqrt\eps/(2\pi)^{5/2}\le1$. Substitution in \eqref{eq:monotonicity} then makes its bracket nonnegative. For $|\omega|=1$, the right-hand side is approximately $28.833655$.

Neither \eqref{eq:threshold} nor \eqref{eq:algebraic-threshold} by itself proves that an initial smallness condition is preserved. That role belongs to \eqref{eq:epsilon1}--\eqref{eq:epsilon2} and Corollary~\ref{cor:barrier}. The global threshold $\eps_2$ remains dependent on the nonnumerical constant $c_\omega$ in \eqref{eq:gradient}; this paper does not claim a fully numerical global convergence threshold.

\vspace{0.3cm}

{\small{\bf{Acknowledgments:}}
 Our main result had been reported in Henan University in October 16th, 2025 and we would like to thank Prof. H.Q. Liu for the invitation. We also thank Mr. H.Y. Chen for the discussion about \eqref{eq:algebraic-threshold}.

\section*{Declarations}
\subsection*{Funding}

\subsection*{Conflicts of interests}
The authors have no relevant financial or non-financial interests to disclose.
\subsection*{Availability of data and material}
Not applicable.
\subsection*{Code availability}
Not applicable.

\end{document}